\documentclass{amsart}
\usepackage{geometry}                
\usepackage{graphicx}
\usepackage{amssymb}
\usepackage{epstopdf,amsmath}
\newtheorem{prop}{Proposition}[section]
\newtheorem{lemma}[prop]{Lemma}

\newtheorem{theorem}[prop]{Theorem}

\DeclareMathOperator{\Poly}{Poly}
\DeclareMathOperator{\Deg}{Deg}

\newcommand{\eps}{\epsilon}

\newcommand{\RR}{\mathbb{R}}
\newcommand{\cZ}{\mathcal{Z}}
\newcommand{\tcZ}{\tilde{\mathcal{Z}}}

\title{Distinct distance estimates and low degree polynomial partitioning}
\author{Larry Guth}

\begin{document}

\maketitle

\begin{abstract} We give a shorter proof of a slightly weaker version of a theorem from \cite{GK2}: we prove that if $\frak L$ is a set of $L$ lines in $\RR^3$ with at most $L^{1/2}$ lines in any low degree algebraic surface, then the number of $r$-rich points of $\frak L$ is $\lesssim L^{(3/2) + \epsilon} r^{-2}$.  This result is one of the main ingredients in the proof of the distinct distance estimate in \cite{GK2}.  With our slightly weaker theorem, we get a slightly weaker distinct distance estimate: any set of $N$ points in $\RR^2$ determines at least $c_\epsilon N^{1 - \epsilon}$ distinct distances.

\end{abstract}

In \cite{E}, Erd{\H o}s asked how few distinct distances may be determined by a set of $N$ points in the plane.  
He conjectured that a square grid of points is near-optimal, giving a conjectural lower bound of $c N (\log N)^{-1/2}$.  Quite recently, in \cite{ES}, Elekes and Sharir suggested a new approach to this problem, connecting it to the incidence geometry of curves in 3-dimensional space.  This approach was carried out by Katz and the author in \cite{GK2}, proving that any set of $N$ points determines $\ge c N (\log N)^{-1}$ distinct distances.  In this paper, we give a variation of the most difficult step of the proof.  We will prove a slightly weaker result, but using a shorter argument.

The main work in \cite{GK2} is an estimate about lines in $\RR^3$.  If $\frak L$ is a set of lines in $\RR^3$, then a point $x$ is called $r$-rich if it lies in at least $r$ lines of $\frak L$.  We write $P_r(\frak L)$ for the set of $r$-rich points of $\frak L$.

\begin{theorem} \label{gkthm} (Theorem 1.2 in \cite{GK2}) If $\frak L$ is a set of $L$ lines in $\RR^3$ with at most $L^{1/2}$ lines in any plane or regulus, and if $2 \le r \le L^{1/2}$, then $|P_r(\frak L)| \le C L^{3/2} r^{-2}$.  
\end{theorem}

The distinct distance estimate follows from combining the approach of Elekes and Sharir with this bound.  The proof of Theorem \ref{gkthm} is somewhat involved.  There are different arguments for the cases $r=2$ and $r \ge 3$ and each argument is pretty long.  The case $r \ge 3$ uses the idea of polynomial partitioning, which will also be central to this paper.  The case $r=2$ uses the theory of ruled surfaces.  We will prove a slightly weaker result using only polynomial partitioning.  

\begin{theorem} \label{mainincid} For any $\epsilon > 0$, there are $D(\epsilon)$, $K(\epsilon)$ so that the following holds.

If $\frak L$ is a set of $L$ lines in $\RR^3$, and there are less than $ L^{(1/2) + \epsilon}$ lines of $\frak L$ in any irreducible algebraic surface of degree at most $D$, and if $2 \le r \le 2 L^{1/2}$, then

$$ | P_r(\frak L) | \le K L^{(3/2) + \epsilon} r^{-2}. $$

\end{theorem}

Using Theorem \ref{mainincid} in place of Theorem \ref{gkthm} in the arguments of \cite{GK2}, one gets the following slightly weaker distinct distance estimate.

\begin{theorem} \label{dd1-eps} For any $\epsilon > 0$, there is a constant $c_\epsilon > 0$ so that any set of $N$ points in the plane determines at least $c_\epsilon N^{1 - \epsilon}$ distinct distances.
\end{theorem}

Polynomial partitioning is one of the main new ideas in \cite{GK2}, and it will also be the key tool in our proof.
We recall the statement of the partitioning theorem.

\begin{theorem} \label{polyhamintro} (Theorem 4.1 in \cite{GK2}) For each dimension $n$ and each degree $D \ge 1$, the following holds.  For any finite set $S \subset \RR^n$, we can find a non-zero polynomial $P$ of degree at  most $D$ so that $\RR^n \setminus Z(P)$ is a union of disjoint open sets $O_i$, and for each of these sets,

$$ |S \cap O_i | \le C_n D^{-n} |S|.$$

\end{theorem}

This polynomial partitioning result is a corollary of the Stone-Tukey ham sandwich theorem \cite{ST}.  Polynomial partitioning is useful in divide and conquer arguments.  The set $S$ is divided into a part in each cell $O_i$ plus a part in a lower-dimensional surface $Z(P)$.  In a divide and conquer argument, we estimate each of these contributions separately and then add up the results.  

Kaplan, Matou\u{s}ek, and Sharir wrote a paper \cite{KMS} on the polynomial partitioning technique.  They give a good exposition of the topic.  They show how to use polynomial partitioning to give new proofs of some classical results in incidence geometry, such as the Szemer\'edi-Trotter theorem.  They also discuss how polynomial partitioning compares with other partitioning methods, such as the cutting method (see Section 2.3 of \cite{KMS}).  

The arguments of \cite{GK2} use polynomial partitioning with degree $D$ equal to a power of $L$.  This gives good bounds on what happens in the cells $O_i$, but it also makes $Z(P)$ rather complicated.  In \cite{SolTao}, Solymosi and Tao gave a modification of this argument using partitioning with degree $D$ equal to a large constant, and using induction to control what happens in each cell.  In \cite{SS}, Sharir and Solomon further developed this method, proving estimates for lines in $\RR^4$.  We will use this low degree partitioning method to prove Theorem \ref{mainincid}.  

Here is the main new issue that arises in the proof of Theorem \ref{mainincid}.  Recall that we use a low degree polynomial to partition $\RR^3$ into cells $O_i$, and we plan to use induction to study the behavior of the lines entering each cell.  Let $\frak L_i$ denote the lines of $\frak L$ that intersect the cell $O_i$.  By hypothesis, we know that $\frak L$ contains less than $| \frak L|^{(1/2) + \epsilon}$ lines in any low degree surface.  Since $\frak L_i \subset \frak L$, $\frak L_i$ contains less than $| \frak L|^{(1/2) + \epsilon}$ lines in any low degree surface.  But that doesn't mean that $\frak L_i$ contains less than $| \frak L_i|^{(1/2) + \epsilon}$ lines in any low degree surface.  Therefore, we cannot immediately apply induction to $\frak L_i$.  At first sight, the inductive argument doesn't look like it will close.  The main new ingredient in this paper is a way to organize the low degree surfaces containing many lines.  By keeping track of their contribution, we can make the induction close.

\vskip5pt

{\it Acknowledgements.} I would like to thank Nets Katz for many interesting discussions about these ideas over the last several years.  I would also like to thank the referee for helpful suggestions about the exposition.

\section{Background and notation}

Our proof is based on polynomial partitioning.  Here we restate the partitioning theorem with an extra condition bounding the number of cells $O_i$.

\begin{theorem} \label{polyham} For each dimension $n$ and each degree $D \ge 1$, the following holds.  For any finite set $S \subset \RR^n$, we can find a non-zero polynomial $P$ of degree at  most $D$ so that $\RR^n \setminus Z(P)$ is a union of disjoint open sets $O_i$ obeying the following:

\begin{itemize}

\item For each $i$, $|S \cap O_i | \le C_n D^{-n} |S|$.

\item The number of open sets $O_i$ is at most $C_n D^n$.  

\end{itemize}

\end{theorem}

\begin{proof}  The first claim is Theorem 4.1 in \cite{GK2}.  So we just need to prove the second claim.

The number of connected components of the complement $\RR^n \setminus Z(P)$ is at most $C_n D^n$, by estimates proven independently by Oleinik-Petrovsky \cite{OP}, Milnor \cite{Mi}, and Thom \cite{Th}.  A short proof was also given by Solymosi and Tao, as Theorem A.1. in their paper \cite{SolTao}.  This implies the second claim.

However, we don't need to appeal to these results.  The statement of the theorem does not require that each open set $O_i$ is connected.  By Theorem 4.1 of \cite{GK2}, we can write $\RR^n \setminus Z(P)$ as a union of open sets $U_j$ with $|S \cap U_j | \le C_n D^{-n} |S|$.  We can then define each $O_i$ to be a union of some of the $U_j$ so that each $O_i$ contains $\le C_n D^{-n} |S|$ points of $S$ and the number of sets $O_i$ is at most $C_n D^n$. \end{proof}

We will also need a version of the B\'ezout theorem.  The simplest version of the B\'ezout theorem is the following.

\begin{theorem} \label{bezout} If $P, Q$ are non-zero polynomials in $\RR[x_1, x_2]$ with no common factor, then $Z(P) \cap Z(Q) \subset \RR^2$ contains at  most $(\Deg P) (\Deg Q)$ points.
\end{theorem}

We need a version of this theorem for polynomials in three variables where we count the number of lines in $Z(P) \cap Z(Q)$.

\begin{theorem} \label{bezoutlines} If $P, Q$ are non-zero polynomials in $\RR[x_1, x_2, x_3]$ with no common factor, then $Z(P) \cap Z(Q)$ contains at most $(\Deg P) (\Deg Q)$ lines.
\end{theorem}

Proofs of these classical results appear in \cite{GK}.  They are Corollaries 2.3 and 2.4.  See also Section 2 of \cite{EKS} for a proof of Theorem \ref{bezoutlines} and a review of related material.  A more general version of Theorem 1.2 can be found in van der Waerden's book {\it Modern Algebra}, \cite{VW}, Volume 2, page 16.  

We will also use the Szemer\'edi-Trotter theorem, which we record here in the following form:

\begin{theorem} \label{szemtrot} (\cite{SzTr}) If $\frak L$ is a set of $L$ lines in $\RR^n$, then 

$$ |P_r(\frak L) | \le C \left( L^2 r^{-3} + L r^{-1} \right). $$

\end{theorem}

There are several nice proofs of the Szemer\'edi-Trotter theorem that have appeared since the original article.  In \cite{CEGSW}, Clarkson et al. gave a proof using the method of cuttings.  In \cite{Sz}, Sz\'ekely gave a proof using the crossing number lemma.  In \cite{KMS}, Kaplan, Matou\u{s}ek, and Sharir gave a proof using the polynomial partitioning theorem.  Their proof is closely related to the ideas in this paper.

We end with a note on constants.  We will use $C$ to denote a constant that may change from line to line.  If we want to label a particular constant to refer to later, we will call it $C_1, C_2,$ etc.

\section{A stronger result for inductive purposes}

We will prove Theorem \ref{mainincid} by induction.  To make the induction work, we prove a slightly stronger result.  The stronger result says that for any set of lines $\frak L$ in $\RR^3$, there is a small set of low degree surfaces that account for all but $\sim L^{(3/2) + \epsilon} r^{-2}$ of the $r$-rich points of $\frak L$.  

To state our theorem we need a piece of notation.  If $\frak L$ is a set of lines and $Z$ is an algebraic surface, we define $\frak L_Z \subset \frak L$ to be the set of lines of $\frak L$ that lie in $Z$.

\begin{theorem} \label{clusterinZr} For any $\epsilon > 0$, there are $D(\epsilon)$, and $K(\epsilon)$ so that the following holds.
For any $ r \ge 2$, let $r' = \lceil (9/10) r \rceil$, the least integer which is at least $(9/10) r$.  

If $\frak L$ is a set of $L$ lines in $\RR^3$, and if $2 \le r \le 2 L^{1/2}$, then there is a set $\mathcal{Z}$ of algebraic surfaces so that

\begin{itemize}

\item Each surface $Z \in \mathcal{Z}$ is an irreducible surface of degree at most $D$.

\item Each surface $Z \in \mathcal{Z}$ contains at least $L^{(1/2) + \epsilon}$ lines of $\frak L$.

\item $|\mathcal{Z}| \le 2 L^{(1/2) - \epsilon}$.

\item $| P_r(\frak L) \setminus \cup_{Z \in \mathcal{Z}} P_{r'}(\frak L_Z) | \le K L^{(3/2) + \epsilon} r^{-2}$.

\end{itemize}

\end{theorem}

Theorem \ref{clusterinZr} implies Theorem \ref{mainincid}.  If there are less than $L^{(1/2) + \epsilon}$ lines of $\frak L$ in any irreducible algebraic surface of degree at most $D$, then the set $\mathcal{Z}$ must be empty, and so Theorem \ref{clusterinZr} implies that $|P_r(\frak L)| \le K L^{(3/2) + \epsilon} r^{-2}$.  

In our theorems above, we always assumed that $r \le 2 L^{1/2}$.  Studying $r$-rich points for $r > 2 L^{1/2}$ is much simpler.  We recall the following elementary estimate, which will also be useful in our proof.

\begin{prop} \label{bigr}
If $\frak L$ is a set of $L$ lines in $\RR^d$ for $d \ge 2$, and if $r > 2 L^{1/2}$, then $|P_r(\frak L)| \le 2 L r^{-1}$.
\end{prop}

We include the well-known proof here, because it is a model for a different proof below.  

\begin{proof} Let $P_r(\frak L)$ be $\{ x_1, x_2, .., x_M \}$, with $M = |P_r(\frak L)|$.  Now $x_1$ lies in at least $r$ lines of $\frak L$.  The point $x_2$ lies in at least $(r-1)$ lines of $\frak L$ that did not contain $x_1$.  More generally, the point $x_j$ lies in at least $r- (j-1)$ lines of $\frak L$ that did not contain any of the previous points $x_1, ..., x_{j-1}$.  Therefore, we have the following inequality for the total number of lines:

$$ L \ge \sum_{j=1}^M \max( r - j, 0). $$

If $M \ge r/2$, then we would get $L \ge (r/2) (r/2) = r^2/4$.  But by hypothesis, $r > 2 L^{1/2}$, giving a contradiction.  Therefore, $M < r/2$, and we get $L \ge M (r/2)$ which proves the proposition.  \end{proof}

\section{Proof of Theorem \ref{clusterinZr}}

Here is an outline of our proof.  We will use induction on the number of lines in $\frak L$.

First, we use a low degree polynomial partitioning argument to cut $\RR^3$ into cells $O_i$.  For each cell, we use induction to study the lines of $\frak L$ that enter that cell.  For each cell, we get a set of surfaces $\mathcal{Z}_i$ that accounts for all but a few of the $r$-rich points in $O_i$.  Combining these surfaces with the polynomial partitioning surface, we will get a large set of surfaces $\tilde {\mathcal{Z}}$ with the following properties:

\begin{itemize}

\item Each surface $Z \in \tilde {\mathcal{Z}}$ is an irreducible algebraic surface of degree at most $D$.

\item $| \tilde{\cZ} | \le \Poly(D) L^{(1/2) - \epsilon} \log L$.

\item $| P_r(\frak L) \setminus \cup_{Z \in \tilde{\mathcal{Z}}} P_{r'}(\frak L_Z) | \le (1/100) K L^{(3/2) + \epsilon} r^{-2}$.

\end{itemize}  

(We write $A \le \Poly(D) B$ to mean that is an exponent $p$ and a constant $C$ so that $A \le C D^p B$.)

This set of surfaces $\tilde {\cZ}$ does not close the induction.  There are too many surfaces in $\tilde{\cZ}$, and we don't know that each surface contains $L^{(1/2) + \epsilon}$ lines of $\frak L$.  The second step is to prune $\tcZ$.  We will define

$$ \cZ := \{ Z \in \tcZ | Z \textrm{ contains at least } L^{(1/2) + \epsilon} \textrm{ lines of } \frak L \}. $$

\noindent Then we will check that $\cZ $ satisfies the conclusions of the theorem.  First, we will prove that $|\cZ| \le 2 L^{(1/2) - \epsilon}$.  This follows from a simple counting argument, similar to the proof of Proposition \ref{bigr} above.  
Second, we will check that the surfaces in $\tcZ \setminus \cZ$ did not contribute too much to controlling the $r$-rich points of $\frak L$.  More precisely we will prove that

$$ \sum_{Z \in \tcZ \setminus \cZ}  |P_{r'}( \frak L_Z) | \le (1/100) K L^{(3/2) + \epsilon} r^{-2}. $$

\noindent To prove this bound, we use Szemer\'edi-Trotter to bound the size of $P_{r'}(\frak L_Z)$ in terms of $|\frak L_Z|$ for each surface $Z \in \tcZ \setminus \cZ$, and we use a simple counting argument to control how many surfaces $Z$ have large $|\frak L_Z|$.  This finishes our outline.  Now we begin the proof of Theorem \ref{clusterinZr}.   

We remark that if $\epsilon \ge 1/2$ then the theorem is trivial: we can take $\cZ$ to be empty, and it is easy to check that $|P_r(\frak L)| \le 2 L^2 r^{-2}$.  (This follows from Szemer\'edi-Trotter, which gives a stronger estimate.  But it also follows from a simple double-counting argument.)  So we can assume that $\epsilon \le 1/2$.  

We start by discussing how to choose $D = D(\eps)$ and $K = K(\eps)$.  We will choose $D$ a large constant depending on $\epsilon$ and then we will choose $K$ a large constant depending on $\epsilon$ and $D$.  As long as these are large enough at certain points in the proof, the argument goes through.  For example, we will choose $K$ large enough that

\begin{equation} \label{klarge}
K \ge 10 (2 D)^{2 / \eps}. 
\end{equation}

The proof is by induction on $L$.  We start by checking the base of the induction.  Because of equation \ref{klarge}, we claim the theorem holds when $L^\eps \le 2D$.  Suppose that $\frak L$ is a set of $L$ lines with $L^\eps \le 2 D$, and that $2 \le r \le 2 L^{1/2}$.  We choose $\cZ$ to be the empty set.  Using equation \ref{klarge}, we see that

$$ |P_r (\frak L)| \le L^2 \le (2D)^{2/\eps} \le K / 10 \le K L^{(3/2) + \eps} r^{-2}. $$

We have now established the base of the induction.  By the inductive hypothesis, we can assume that the theorem holds for sets of at most $L/2$ lines.

\subsection{Building $\tcZ$}

Let $S$ be any subset of $P_r(\frak L)$.  An important case is $S = P_r(\frak L)$, but we will have to consider other sets as well.  We use Theorem \ref{polyham} to do a polynomial partitioning of the set $S$ with a polynomial of degree at most $D$.  The polynomial partitioning theorem, Theorem \ref{polyham}, says that there is a non-zero polynomial $P$ of degree at most $D$ so that 

\begin{itemize}

\item $\RR^3 \setminus Z(P)$ is the union of at most $C D^3$ disjoint open cells $O_i$, and

\item for each cell $O_i$, $|S \cap O_i | \le C D^{-3} |S|$.

\end{itemize}

We define $\frak L_i \subset \frak L$ to be the set of lines from $\frak L$ that intersect the open cell $O_i$.  We note that $S \cap O_i  \subset P_r(\frak L_i)$.  If a line does not lie in $Z(P)$, then it can have at most $D$ intersection points with $Z(P)$, which means that it can enter at most $D+1$ cells $O_i$.  So each line of $\frak L$ intersects at most $D+1$ cells $O_i$.  Therefore, we get the following inequality:

\begin{equation} \label{totalgammair}
\sum_i |\frak L_i| \le (D+1) L \le 2 D L.
\end{equation}

Let $\beta > 0$ be a large parameter that we will choose below.  We say that a cell $O_i$ is $\beta$-good if

\begin{equation}\label{betagoodr}
|\frak L_i| \le \beta D^{-2} L.
\end{equation}

The number of $\beta$-bad cells is at most $2 \beta^{-1} D^3$.  Each cell contains at most $C D^{-3} |S|$ points of $S$.  Therefore, the bad cells all together contain at most $C \beta^{-1} |S|$ points of $S$.  We now choose $\beta$ so that $C \beta^{-1} \le (1/100)$.  $\beta$ is an absolute constant, independent of $\epsilon$.  We now have the following estimate:

\begin{equation}\label{badcellsboundr}
\textrm{The union of the bad cells contains at most $(1/100) |S|$ points of $S$.}
\end{equation}

For each good cell $O_i$, we apply induction to understand $\frak L_i$.  By choosing $D$ sufficiently large, we can guarantee that for each good cell, $| \frak L_i | \le (1/2) L$.  Now there are two cases, depending on whether $r \le 2 |\frak L_i|^{1/2}$.  

If $r \le 2 |\frak L_i|^{1/2}$, then we can apply the inductive hypothesis.  In this case, we see that there is a set $\cZ_i$ of irreducible algebraic surfaces of degree at most $D$ with the following two properties:

\begin{equation}\label{sizeofgamma_ir}
| \cZ_i| \le 2 |\frak L_i|^{(1/2) - \epsilon} \le 2 (\beta D^{-2} L)^{(1/2) - \epsilon}.
\end{equation}

Because $S \cap O_i \subset P_r(\frak L_i)$, we also get:

\begin{equation}\label{missedpointsinO_ir}
| (S \cap O_i) \setminus \cup_{Z \in \cZ_i} P_{r'}( \frak L_Z) | \le K |\frak L_i|^{(3/2) + \epsilon} r^{-2} \le K (\beta D^{-2} L)^{(3/2) + \epsilon} r^{-2} \le C_1 K D^{-3 - 2 \epsilon} L^{(3/2) + \epsilon} r^{-2}.
\end{equation}

On the other hand, if $r > 2 |\frak L_i|^{1/2}$, then we define $\cZ_i$ to be empty, and Proposition \ref{bigr} gives the bound 

\begin{equation}\label{missedpointsinO_irbigr}
|S \cap O_i| \le |P_r( \frak L_i) | \le 2 |\frak L_i| r^{-1} \le 2 L r^{-1} \le 4 L^{3/2} r^{-2}. 
\end{equation}

By choosing $K$ sufficiently large compared to $D$, we can arrange that $4 L^{3/2} r^{-2} \le C_1 K D^{-3 - 2 \epsilon} L^{(3/2) + \epsilon} r^{-2}$.  Therefore, inequality \ref{missedpointsinO_ir} holds for the good cells with $r > 2 |\frak L_i|^{1/2}$ as well as the good cells with $r \le 2 |\frak L_i|^{1/2}$.  
We sum this inequality over all the good cells:

$$ \sum_{O_i \textrm{ good}} |(S \cap O_i) \setminus \cup_{Z \in \cZ_i} P_{r'}( \frak L_Z) | \le C D^3 \cdot C_1 K D^{-3 - 2 \epsilon} L^{(3/2) + \epsilon} r^{-2} \le C_2 D^{-2 \epsilon} K L^{(3/2) + \epsilon} r^{-2} . $$

We choose $D(\epsilon)$ large enough so that $C_2 D^{-2 \epsilon} \le (1/400)$.  Therefore, we get the following:

\begin{equation}\label{goodcellsboundr} \sum_{O_i \textrm{ good}} |(S \cap O_i) \setminus \cup_{Z \in \cZ_i} P_{r'}( \frak L_Z) | \le (1/400) K L^{(3/2) + \epsilon} r^{-2}. 
\end{equation}

We have studied the points of $S$ in the good cells.  Next we study the points of $S$ in the zero set of the partioning polynomial $Z(P)$.  Let $Z_j$ be an irreducible component of $Z(P)$.  If $x \in S \cap Z_j$, but $x \notin P_{r'}(\frak L_{Z_j})$, then $x$ must be contained in at least $r/10$ lines of $\frak L \setminus \frak L_{Z_j}$.  Each line of $\frak L$ that is not contained in $Z_j$ has at most $\Deg(Z_j)$ intersection points with $Z_j$.  Therefore, 

$$| (S \cap Z_j) \setminus P_{r'}(\frak L_{Z_j})| \le 10 r^{-1} (\Deg Z_j) L. $$

If $\{ Z_j \}$ are all the irreducible components of $Z(P)$, then we see that

$$ | (S \cap Z(P)) \setminus \cup_j P_{r'}(\frak L_{Z_j}) | \le 10 r^{-1} D L. $$

We choose $K = K(\epsilon, D)$ sufficiently large so that $10 D \le (1/800) K$.   Since $r \le 2 L^{1/2}$, we have

\begin{equation}\label{cellwallsboundr}
| (S \cap Z(P)) \setminus \cup_j P_{r'}(\frak L_{Z_j}) | \le (1/800) K L r^{-1} \le (1/400) K L^{3/2} r^{-2}.  
\end{equation}

Now we define $\tcZ_S$ to be the union of $\cZ_i$ over all the good cells $O_i$ together with all the irreducible components $Z_j$ of $Z(P)$.   Each surface in $\tcZ_S$ is an algebraic surface of degree at most $D$.  By equation \ref{sizeofgamma_ir}, we have the following estimate for $| \tcZ_S|$: 

\begin{equation}\label{size of Z_Sr}
 | \tcZ_S | \le C D^3 (\beta D^{-2} L)^{(1/2) - \epsilon} + D \le \Poly(D) L^{(1/2) - \epsilon}.
\end{equation}

Summing the contribution of the bad cells in equation \ref{badcellsboundr}, the contribution of the good cells in equation \ref{goodcellsboundr}, and the contribution of 
the cell walls in equation \ref{cellwallsboundr}, we get:

\begin{equation}\label{totalboundr}
| S \setminus \cup_{Z \in \tcZ_S} P_{r'}(\frak L_Z) | \le (1/100) |S| + (1/200) K L^{(3/2) + \epsilon} r^{-2}.
\end{equation}

If we didn't have the $(1/100) |S|$ term coming from the bad cells, we could simply take $S = P_r(\frak L)$ and $\tcZ = \tcZ_S$.  Because of this term, we need to run the above construction repeatedly.

Let $S_1 = P_r(\frak L)$, and let $\tcZ_{S_1}$ be the set of surfaces constructed above.  Now we define
$ S_2 = S_1 \setminus \cup_{Z \in \tcZ_{S_1}} P_{r'} (\frak L_Z)$.  We iterate this procedure, defining

$$ S_{j+1} := S_j \setminus \cup_{Z \in \tcZ_{S_j}} P_{r'} (\frak L_Z). $$  

Each set $S_j$ is a subset of $P_r(\frak L)$.  Each set of surfaces $\tcZ_{S_j}$ has cardinality at most $\Poly(D) L^{(1/2) - \epsilon}$.  
Iterating equation \ref{totalboundr} we see:

\begin{equation}\label{iterationr}
|S_{j+1} | \le (1/100) |S_j| + (1/200) K L^{(3/2) + \epsilon} r^{-2}.
\end{equation}

We define $J = C \log L$ for a large constant $C$.  Because of the iterative formula in equation \ref{iterationr}, we get 

\begin{equation}\label{endboundr}
|S_J| \le (1/100) K L^{(3/2) + \epsilon} r^{-2}.
\end{equation}

We define $\tcZ = \cup_{j=1}^{J-1} \tcZ_{S_j}$.  This set of surfaces has the following properties.  Since each set $\tcZ_{S_j}$ has at most $\Poly(D) L^{(1/2) - \epsilon}$ surfaces, we get:

\begin{equation}\label{sizeoftcZr}
|\tcZ | \le \Poly(D) L^{(1/2) - \epsilon} \log L.
\end{equation}

Also, $P_r(\frak L) \setminus \cup_{Z \in \tcZ} P_{r'}(\frak L_Z) = S_J$, and so equation \ref{endboundr} gives:

\begin{equation}\label{tcZboundr}
| P_r(\frak L) \setminus \cup_{Z \in \tcZ} P_{r'}(\frak L_Z) | \le (1/100) K L^{(3/2) + \epsilon} r^{-2}.
\end{equation}

This finishes our construction of $\tcZ$.  Next we prune $\tcZ$ down to our desired set of surfaces $\cZ$.

\subsection{Pruning $\tcZ$}

We define

$$ \cZ := \{ Z \in \tcZ | Z \textrm{ contains at least } L^{(1/2) + \epsilon} \textrm{ lines of } \frak L \}. $$

To close our induction, we have to check two properties of $\cZ$.

\begin{enumerate}

\item $|\mathcal{Z}| \le 2 L^{(1/2) - \epsilon}$.

\item $| P_r(\frak L) \setminus \cup_{Z \in \mathcal{Z}} P_{r'}(\frak L_Z) | \le K L^{(3/2) + \epsilon} r^{-2}$.

\end{enumerate}

We begin with a simple lemma about surfaces that each contain many lines.

\begin{lemma} \label{surfcountr} Suppose $\frak L$ is a set of lines in $\RR^3$, and $\mathcal{Y}$ is a set of irreducible algebraic surfaces of degree at most $D$, and suppose that each surface $Z \in \mathcal{Y}$ contains at least $A$ lines of $\frak L$.  

If $A > 2 D | \frak L |^{1/2}$, then $| \mathcal{Y} | \le 2 | \frak L | A^{-1}$. 

\end{lemma}

\begin{proof} The proof of this lemma follows the same idea as the proof of Proposition \ref{bigr}.  By the B\'ezout theorem for lines, Theorem \ref{bezoutlines}, the intersection of any two surfaces $Z_1, Z_2 \in \mathcal{Y}$ contains at most $D^2$ lines of $\frak L$.

We choose an ordering of the surfaces of $\mathcal{Y}$.  We consider the surfaces one at a time in order and count the number of new lines.

$Z_1$ contains at least $A$ lines of $\frak L$.
$Z_2$ contains at least $A - D^2$ lines of $\frak L$ that are not in $Z_1$.
$Z_{j+1}$ contains at least $A - jD^2$ lines of $\frak L$ that are not in the previous surfaces $Z_1, ..., Z_j$.  Therefore, we get the following inequality:

$$ |\frak L| \ge \sum_{j=1}^{|\mathcal{Y}|} \max( A - j D^2, 0 ). $$

If $j \le (1/2) A D^{-2}$, then $A - j D^2 \ge A/2$.  Therefore, if $|\mathcal{Y}| \ge (1/2) A D^{-2}$, then we see that $|\frak L| \ge (1/2) A D^{-2} ( A/2)$.  By hypothesis, we know $A > 2 D | \frak L|^{1/2}$, which gives the contradiction $| \frak L | > |\frak L|$.  Therefore, $|\mathcal{Y}| \le (1/2) A D^{-2}$.  Now we see that $ | \frak L | \ge |\mathcal{Y} | (A/2)$, and this completes the proof of the lemma. \end{proof}

We apply this lemma with $\mathcal{Y} = \cZ$ and $A = L^{(1/2) + \epsilon}$.  We can assume that $L^{\epsilon} > 2 D$, because the case of $L^\epsilon \le 2 D$ was the base of our induction, and we handled it by choosing $K$ sufficiently large.  Therefore, $A = L^{(1/2) + \epsilon} > 2 D L^{1/2}$, and the hypotheses of Lemma \ref{surfcountr} are satisfied.  The lemma tells us that $|\cZ| \le 2 L^{(1/2) - \epsilon}$, which proves item (1) above.  Now we turn to item (2).  We recall equation \ref{tcZboundr}:

$$| P_r(\frak L) \setminus \cup_{Z \in \tcZ} P_{r'}(\frak L_Z) | \le (1/100) K L^{(3/2) + \epsilon} r^{-2}. $$

Therefore, it suffices to check that

$$ \sum_{Z \in \tcZ \setminus \cZ} |P_{r'}( \frak L_Z) | \le (1/100) K L^{(3/2) + \epsilon} r^{-2}. $$

We sort $\tcZ \setminus \cZ$ according to the number of lines in each surface.  For each integer $s \ge 0$, we define:

$$ \tcZ_s := \{ Z \in \tcZ \textrm{ so that } | \frak L_Z | \in  [2^{s}, 2^{s+1}) \}. $$

Since each surface of $\tcZ$ with at least $L^{(1/2) + \epsilon}$ lines of $\frak L$ lies in $\cZ$, we see that:

\begin{equation}\label{uniontczsr}
\tcZ \setminus \cZ \subset \bigcup_{2^s \le L^{(1/2) + \epsilon}} \tcZ_s.
\end{equation}

For each $Z \in \tcZ_s$, $|\frak L_Z| \le 2^{s+1}$.  We use the Szemer\'edi-Trotter theorem, Theorem \ref{szemtrot}, to bound $P_{r'}(\frak L_Z)$.  Since $r' \ge (9/10) r$, Szemer\'edi-Trotter gives:

\begin{equation}\label{stboundr}
P_{r'} ( \frak L_Z) \le C \left( 2^{2s} r^{-3} + 2^s r^{-1} \right).
\end{equation}

Using Lemma \ref{surfcountr} with $A = 2^s$, we get the following estimate for $| \tcZ_s |$:

\begin{equation}\label{sizetcZ_sr}
\textrm{If } 2^s > 2 D L^{1/2}, \textrm{ then } |\tcZ_s| \le 2 L 2^{-s}.  
\end{equation}

We can now estimate $ \sum_{Z \in \tcZ \setminus \cZ} |P_{r'}( \frak L_Z)|$.

\begin{equation}\label{groupbysr}
 \sum_{Z \in \tcZ \setminus \cZ} |P_{r'}( \frak L_Z)| \le \sum_{2^s \le L^{(1/2) + \epsilon}} \left( \sum_{Z \in \tcZ_s} |P_{r'}(\frak L_Z)| \right) \le C \sum_{2^s \le L^{(1/2) + \epsilon}}  |\tcZ_s| \left( 2^{2s} r^{-3} + 2^s r^{-1} \right). 
\end{equation}

We consider the contribution to the last sum from $s$ in the range $2 D L^{1/2} < 2^s \le L^{(1/2) + \epsilon}$.  Using equation \ref{sizetcZ_sr} to estimate $|\tcZ_s|$ gives:

$$ \sum_{2 D L^{1/2} < 2^s \le L^{(1/2) + \epsilon}} |\tcZ_s| \left( 2^{2s} r^{-3} + 2^s r^{-1} \right) \le \sum_{2^s \le L^{(1/2) + \epsilon}} 
(2 L 2^{-s}) \left( 2^{2s} r^{-3} + 2^s r^{-1} \right) \le $$

$$\le C \sum_{2^s \le L^{(1/2) + \epsilon}} (L 2^s r^{-3} + L r^{-1}) \le C( L^{(3/2) + \epsilon} r^{-3} + L (\log L) r^{-1} ) \le 
C L^{(3/2) + \epsilon} r^{-2}. $$

Next we consider the contribution to the last sum in equation \ref{groupbysr} from $s$ in the range $2^s \le 2 D L^{1/2}$.  In this range of $s$, we use Equation \ref{sizeoftcZr} to bound $| \tcZ_s|$: $| \tcZ_s|  \le |\tcZ| \le  \Poly(D) L^{(1/2) - \epsilon} \log L$.  

\begin{equation} \label{smallscon} \sum_{2^s \le 2 D L^{1/2}} |\tcZ_s|  \left( 2^{2s} r^{-3} + 2^s r^{-1} \right) \le \Poly(D) \left(  L^{(1/2) - \epsilon} \log L \right) \left( 2^{2s} r^{-3} + 2^s r^{-1} \right). \end{equation}

Since $2^{s} \le 2 D L^{1/2}$ we see that $2^{2s} r^{-3} \le \Poly(D) L r^{-3}$ and $2^s r^{-1} \le \Poly(D) L^{1/2} r^{-1} \le \Poly(D) L r^{-2}$.  Plugging these into the right-hand side of equation \ref{smallscon}, we get

$$ \sum_{2^s \le 2 D L^{1/2}} |\tcZ_s|  \left( 2^{2s} r^{-3} + 2^s r^{-1} \right) \le \Poly(D) L^{3/2} r^{-2}. $$

All together, we see

$$ \sum_{Z \in \tcZ \setminus \cZ} |P_{r'}( \frak L_Z) | \le \Poly(D) L^{(3/2) + \epsilon} r^{-2}. $$

Choosing $K = K(\epsilon, D)$ sufficiently large, we see that

$$ \sum_{Z \in \tcZ \setminus \cZ} |P_{r'}( \frak L_Z) | \le (1/100) K L^{(3/2) + \epsilon} r^{-2}. $$

This proves item (2), closing the induction, and finishing the proof of Theorem \ref{clusterinZr}.

\section{Distinct distances}

In \cite{ES}, Elekes and Sharir proposed a new approach to the distinct distance problem, connecting it to incidence estimates about curves in $\RR^3$.  A tiny modification of these ideas is explained in Section 2 of \cite{GK2}, connecting the distinct distance problem to an estimate about incidences of lines in $\RR^3$.  The paper \cite{GK2} then uses Theorem \ref{gkthm} to control these incidences.  We can also use our slightly weaker Theorem \ref{mainincid} to prove a slightly weaker bound on the number of distinct distances.

In this section, we give a concise review of the Elekes-Sharir approach to the distinct distance problem.  Using our incidence bound, Theorem \ref{mainincid}, we prove the following distinct distance bound.

\begin{theorem} \label{distbound} For any $\eps > 0$, there is a constant $c_\eps > 0$ so that the following holds.  If $P$ is a set of $N$ points in $\RR^2$, then $P$ determines at least $c_\eps N^{1 - \eps}$ distinct distances.
\end{theorem}

If $P \subset \RR^2$ is a set of points, we let $d(P)$ be the set of distinct distances:

$$ d(P) := \{ | p_1 - p_2| \}_{p_1, p_2 \in P}. $$

The approach of Elekes and Sharir involves the set of distance quadruples $Q(P)$:

$$ Q(P) := \{ (p_1, p_2, p_3, p_4) \in P^4 \textrm{ so that } |p_1 - p_2| = |p_3 - p_4| \not= 0 \}. $$

A simple Cauchy-Schwarz inequality proves the following estimate (Lemma 2.1 in \cite{GK2}):

\begin{equation} \label{dvsQ} |d(P)| \ge \frac{ N^4 - 2 N^3} { |Q(P)|}. \end{equation}

The heart of the matter is to prove an upper bound for $|Q(P)|$.  The next step is to introduce a family of lines in $\RR^3$, $\frak L(P)$, associated to the set $P \subset \RR^2$.  The incidence geometry of this family of lines encodes the distance quadruples.

For any two points $p_1, p_2 \in \RR^2$, we define a line $l_{p_1, p_2} \subset \RR^3$ as follows.  Suppose that $p_1 = (x_1, y_1)$ and $p_2 = (x_2, y_2)$.  We use $x, y, z$ for the coordinates of $\RR^3$.  Then $l_{p_1,p_2}$ is the line defined by the following equations:

\begin{equation} \label{eqlinex} 
2 x = (x_1 + x_2) + (y_1 - y_2) z.
 \end{equation}

\begin{equation} \label{eqliney} 
2y = (y_1 + y_2) + (x_2 - x_1) z.
 \end{equation}

The set $\frak L(P)$ is defined to be $\{ l_{p_1, p_2} \}_{p_1, p_2 \in P}$.  If $P$ is a set of $N$ points, then $\frak L(P)$ is a set of $N^2$ lines.  The connection between $Q(P)$ and $\frak L(P)$ appears in the following lemma.

\begin{lemma} \label{quadinter} A quadruple $(p_1, p_2, p_3, p_4) \in P^4$ is a distance quadruple if and only if
the line $l_{p_1, p_3}$ and the line $l_{p_2, p_4}$ are intersecting or parallel.
\end{lemma}

Remark: The condition of being intersecting or parallel is natural from the projective point of view.  Two lines $l, \bar l$ are intersecting or parallel in $\RR^n$ if and only if they intersect in $\mathbb{RP}^n$.

We now give a proof by direct computation.  The paper \cite{ES} gives a nice motivation for introducing these lines.  The motivation comes from the group of rigid motions of the plane, which is a symmetry group of the distinct distance problem.  This point of view is also explained in Section 2 of \cite{GK2}.  Lemma \ref{quadinter} is proven in Section 2 of \cite{GK2} using the point of view of rigid motions.

\begin{proof} First we describe the projective completion of the line $l_{p_1, p_2}$ in $\mathbb{RP}^3$.  A point in $\mathbb{RP}^3$ is an equivalence class of non-zero vectors $(w,x,y,z) \in \RR^4$, where two vectors are equivalent if one is a scalar multiple of the other.  In these coordinates, the equations for the line $l_{p_1, p_2} \subset \mathbb{RP}^3$ are as follows:

\begin{equation} \label{eqlinexpro} 
2 x = (x_1 + x_2) w + (y_1 - y_2) z.
 \end{equation}

\begin{equation} \label{eqlineypro} 
2y = (y_1 + y_2) w + (x_2 - x_1) z.
\end{equation}

Next we investigate when two lines in $\mathbb{RP}^3$ intersect.  Suppose that $l$ is defined by the equations

\begin{equation} \label{eql}
2x = a_x w + b_x z; 2y = a_y w + b_y z.
\end{equation}

\noindent and $\bar l$ is defined by the equations

\begin{equation} \label{eqbarl}
2x = \bar a_x w + \bar b_x z; 2y = \bar a_y w + \bar b_y z.
\end{equation}

The lines $l$ and $\bar l$ intersect in $\mathbb{RP}^3$ if and only if the following system of two equations in $w,z$ has a non-zero solution:

\begin{equation} \label{lbarlinter}
a_x w + b_x z = \bar a_x w + \bar b_x z ; a_y w + b_y z =  \bar a_y w + \bar b_y z
\end{equation}

\noindent By standard linear algebra, this system of equations has a non-zero solution if and only if an appropriate determinant vanishes, which we can rewrite as the following equation:

\begin{equation} \label{lbarlniceeq}
(a_x - \bar a_x) (b_y - \bar b_y) = (a_y - \bar a_y) (b_x - \bar b_x).
\end{equation}

Now we take $l = l_{p_1, p_3}$ and $\bar l = l_{p_2, p_4}$.  Using equations \ref{eqlinexpro} and \ref{eqlineypro}, we can find the values of $a_x$ etc.  In particular, we see that $a_x = x_1 + x_3$, $a_y = y_1 + y_3$, $b_x = y_1 - y_3$ and $b_y = x_3 - x_1$, and similarly $\bar a_x = x_2 + x_4$, $\bar a_y = y_2 + y_4$, $\bar b_x = y_2 - y_4$, and $\bar b_y = x_4 - x_2$.  When we plug these values into equation \ref{lbarlniceeq}, we get a homogeneous quadratic equation in $x_i$ and $y_i$.  We claim that this equation is equivalent to $(x_1 - x_2)^2 + (y_1 - y_2)^2 = (x_3 - x_4)^2 - (y_3 - y_4)^2$.  Here is the computation.  Plugging the values of $a_x$ etc. into equation \ref{lbarlniceeq}, we immediately get:

$$ \left[ (x_1 + x_3) - (x_2 + x_4) \right] \left[ (x_3 - x_1) - (x_4 - x_2) \right] = \left[ (y_1 + y_3) - (y_2 + y_4) \right] \left[ (y_1 - y_3) - (y_2 -y_4) \right].  $$

Rearranging the terms inside of each large parentheses, this is equivalent to

$$ \left[ (x_3 - x_4) + (x_1 - x_2) \right] \left[ (x_3 - x_4) - (x_1 - x_2) \right] =  \left[  (y_1 - y_2) + (y_3 - y_4) \right] \left[ (y_1 - y_2) - (y_3 - y_4) \right] $$

Expanding both sides, this is equivalent to 

$$ (x_3 - x_4)^2 - (x_1 - x_2)^2 = (y_1 - y_2)^2 - (y_3 - y_4)^2. $$

Moving the negative terms to the other sides, this is equivalent to

$$ (x_3 - x_4)^2 + (y_3 - y_4)^2 = (x_1 - x_2)^2 + (y_1 - y_2)^2. $$

This is equivalent to $|p_3 - p_4| = |p_1 - p_2|$.  

\end{proof}

Because of Lemma \ref{quadinter}, each distance quadruple $(p_1, p_2, p_3, p_4) \in Q(P)$ can be labelled as an intersecting quadruple or a parallel quadruple, depending on whether $l_{p_1, p_3}$ and $l_{p_2, p_4}$ are intersecting or parallel.

The number of parallel quadruples is straightforward to bound.  If $l_{p_1,p_3}$ and $l_{p_2, p_4}$ are parallel, then equations \ref{eqlinex} and \ref{eqliney} imply that $y_1 - y_3 = y_2 - y_4$ and $x_3 - x_1 = x_4 - x_2$.  In other words, $l_{p_1, p_3}$ and $l_{p_2, p_4}$ are parallel if and only if $p_1 - p_2 = p_3 - p_4$.  For any $p_1, p_2, p_3$, there is at most one $p_4 \in P$ so that $p_1 - p_2 = p_3 - p_4$, and so there are at most $N^3$ parallel distance quadruples.  

From now on, we sometimes abbreviate $\frak L(P)$ by $\frak L$.  

The number of intersecting distance quadruples can be counted as follows.  We let $P_{=r}(\frak L)$ denote the set of points that lie in exactly $r$ lines of $\frak L$.  At each point of $P_{=r}(\frak L)$ there are $r^2 - r$ intersecting pairs $(l_1, l_2) \in \frak L^2$.  Therefore, the number of intersecting distance quadruples is

$$ |Q(P)_{inter}| = \sum_{r \ge 2} (r^2 - r) |P_{=r}(\frak L)|. $$

Since $|P_{=r}(\frak L)| = |P_r(\frak L)| - |P_{r+1}(\frak L)|$, we can rewrite this formula as

\begin{equation} \label{|Q|rrich} |Q(P)_{inter}| = \sum_{r \ge 2} (2r - 2) |P_r(\frak L)|. \end{equation}

Therefore, a bound on $|P_r(\frak L)|$ gives a bound on $|Q(P)|$.  

To bound $|P_r(\frak L)|$ the paper \cite{GK2} proves the following result (Proposition 2.8 in \cite{GK2}):

\begin{lemma} \label{nonclustergk} If $P \subset \RR^2$ is a set of $N$ points, then $\frak L(P)$ contains at most $C N$ lines in any plane or regulus, and at most $N$ lines of $\frak L(P)$ contain any point.
\end{lemma}

\noindent With this lemma in hand, \cite{GK2} can apply Theorem \ref{gkthm}, giving the bound $|P_r(\frak L)| \le C N^3 r^{-2}$ for all $2 \le r \le N$.  (And for $r > N+1$, Lemma \ref{nonclustergk} says that $|P_r(\frak L)|  = 0$.)  Plugging these bounds into equation \ref{|Q|rrich} shows that $|Q(P)| \le N^3 + \sum_{r=2}^N C N^3 r^{-1} \le C N^3 \log N$.

We will use Theorem \ref{mainincid} in place of Theorem \ref{gkthm} to give a slightly weaker bound on the number of distance quadruples.  In order to apply Theorem \ref{mainincid} we need a slightly stronger lemma.

\begin{lemma} \label{noncluster} For any degree $D \ge 1$ there is a constant $C_D$ so that the following holds.   If $P \subset \RR^2$ is a set of $N$ points, then $\frak L(P)$ contains at most $C_D N$ lines in any algebraic surface of degree at most $D$.  Also $\frak L(P)$ contains at most $N$ lines that pass through any point.
\end{lemma}

\noindent We will give the proof of Lemma \ref{noncluster} below.  Using Lemma \ref{noncluster}, we can apply Theorem \ref{mainincid}, giving the following bound: for any $\eps > 0$, there is a constant $C_\eps$ so that 

$$|P_r(\frak L)| \le C_\eps N^{3 + \eps} r^{-2}. $$

Plugging this bound into equation \ref{|Q|rrich}, we see that 

$$|Q(P)| \le N^3 + \sum_{r=2}^N (2r - 2) |P_r(\frak L)| \le N^3 + \sum_{r=2}^N C_\eps N^{3 + \eps} r^{-1} \le  C_\eps N^{3 + \eps}. $$
  
Plugging this bound into equation \ref{dvsQ}, we see that $|d(P)| \ge c_\eps N^{1 - \eps}$ for any $\eps > 0$.  This proves Theorem \ref{distbound}.

\subsection{The proof of the non-clustering lemma}

It only remains to prove Lemma \ref{noncluster}.  Suppose that $P \subset \RR^2$ is a set of $N$ points.

We first observe that if $p \in \RR^2$ and $q_1 \not=q_2 \in \RR^2$ then the lines $l_{p, q_1}$ and $l_{p, q_2}$ are skew.  By Lemma \ref{quadinter}, $l_{p, q_1}$ and $l_{p, q_2}$ are non-skew if and only if $| p - p | = |q_1 - q_2|$.  But $|p-p| = 0$ and $|q_1 - q_2| \not= 0$.  

From this observation, we can quickly establish two parts of Lemma \ref{noncluster}.  First, for any plane in $\RR^3$, at most one of the lines $\{ l_{p, q} \}_{q \in P}$ can lie in the plane.  Therefore, any plane contains at most $N$ lines of $\frak L(P)$.  Second, for any point $\RR^3$, at most one of the lines $\{ l_{p, q} \}_{q \in P}$ can contain the point.  Therefore, for any point in $\RR^3$, at most $N$ lines of $\frak L(P)$ contain the point.

Now consider an irreducible polynomial $Q$ with $1 < \Deg Q \le D$.  We will prove that $Z(Q)$ contains $\le 3 D^2 N$ lines of $\frak L(P)$, and this will finish the proof of Lemma \ref{noncluster}.

We let $\frak L_p := \{ l_{p, q} \}_{q \in \RR^2}$.  We would like to understand how many lines of $\frak L_p$ may lie in $Z(Q)$. 

\begin{lemma} \label{onebadpointZ(Q)} If $Q$ is an irreducible polynomial with $1 < \Deg Q \le D$, then there is at most one point $p \in \RR^2$ so that $Z(Q)$ contains at least $2 D^2$ lines of $\frak L_p$.
\end{lemma}

Given Lemma \ref{onebadpointZ(Q)}, we now check that $Z(Q)$ contains at most $3 D^2 N$ lines of $\frak L(P)$.  For $N-1$ of the points $p \in P$, $Z(Q)$ contains at most $2 D^2$ of the lines $\{ l_{p, p'} \}_{p' \in P}$.  For the last point $p \in P$, $Z(Q)$ contains at most all $N$ of the lines $\{ l_{p, p'} \}_{p' \in P}$.  In total, $Z(Q)$ contains at most $(2 D^2 + 1) N$ lines of $\frak L(P)$. 

The proof of Lemma \ref{onebadpointZ(Q)} is based on a more technical lemma which describes the algebraic structure of the set of lines $\{ l_{p,q} \}$ in $\RR^3$.  

\begin{lemma} \label{lemV_p} For each $p$, each point of $\RR^3$ lies in a unique line from the set $\{ l_{p,q} \}_{q \in \RR^2}$.  Moreover, for each $p$, there is a non-vanishing vector field $V_p(x_1,x_2,x_3)$, so that at each point, $V_p(x)$ is tangent to the unique line $l_{p,q}$ through $x$.  Moreover, $V_p(x)$ is a polynomial in $p$ and $x$, with degree at most 1 in the $p$ variables and degree at most 2 in the $x$ variables. 
\end{lemma}

Let us assume this technical lemma for the moment and use it to prove Lemma \ref{onebadpointZ(Q)}.

Fix a point $p \in \RR^2$.  Suppose $Z(Q)$ contains at least $2 D^2$ lines from the set $\frak L_p := \{ l_{p,q} \}_{p,q \in \RR^2}$.  On each of these lines, $Q$ vanishes identically, and $V_p$ is tangent to the line.  Therefore, $V_p \cdot \nabla Q$ vanishes on all these lines.  But $V_p \cdot \nabla Q$ is a polynomial in $x$ of degree at most $2 D - 2$.  If $V_p \cdot \nabla Q$ and $Q$ have no common factor, then the Bezout theorem for lines, Theorem \ref{bezoutlines}, implies that there are at most $2 D^2 - 2D$ lines where the two polynomials vanish.  Therefore, $V_p \cdot \nabla Q$ and $Q$ have a common factor.  Since $Q$ is irreducible, $Q$ must divide $V_p \cdot \nabla Q$, and we see that $V_p \cdot \nabla Q$ vanishes identically on $Z(Q)$.  

Now suppose that $Z(P)$ contains at least $2 D^2$ lines from $\frak L_{p_1}$ and from $\frak L_{p_2}$.  We see that $V_{p_1} \cdot \nabla Q$ and $V_{p_2} \cdot \nabla Q$ vanish on $Z(Q)$.  For each fixed $x$, the expression $V_p \cdot \nabla Q$ is a degree 1 polynomial in $p$.  Therefore, for any point $p$ in the affine span of $p_1$ and $p_2$, $V_p \cdot \nabla Q$ vanishes on $Z(Q)$.  

Suppose that $Z(Q)$ has a non-singular point $x$, which means that $\nabla Q(x) \not= 0$.  In this case, $x$ has a smooth neighborhood $U_x \subset Z(Q)$ where $\nabla Q$ is non-zero.  If $V_p \cdot \nabla Q$ vanishes on $Z(Q)$, then the vector field $V_p$ is a vector field on $U_x$, and so its integral curves lie in $U_x$.  But the integral curves of $V_p$ are exactly the lines of $\frak L_p$.  Therefore, for each $p$ on the line connecting $p_1$ and $p_2$, the line of $\frak L_p$ through $x$ lies in $Z(Q)$.  Since $x$ is a smooth point, all of these lines must lie in the tangent plane $T_x Z(Q)$, and we see that $Z(Q)$ contains infinitely many lines in a plane.  Using Bezout's theorem, Theorem \ref{bezoutlines}, again, we see that $Z(Q)$ is a plane, and that $Q$ is a degree 1 polynomial.  This contradicts our assumption that $\Deg Q > 1$.  

We have now proven Lemma \ref{onebadpointZ(Q)} in the case that $Z(Q)$ contains a non-singular point.  But if every point of $Z(Q)$ is singular, then we get an even stronger estimate on the lines in $Z(Q)$:

\begin{lemma} Suppose that $Q$ is a non-zero irreducible polynomial of degree $D$ on $\RR^3$.  If $Z(Q)$ has no non-singular point, then $Z(Q)$ contains at most $D^2$ lines.
\end{lemma}

\begin{proof} Since every point of $Z(Q)$ is singular, $\nabla Q$ vanishes on $Z(Q)$.  In particular, each partial derivative $\partial_i Q$ vanishes on $Z(Q)$.  We suppose that $Z(Q)$ contains more than $D^2$ lines and derive a contradiction.  Since $\partial_i Q = 0$ on $Z(Q)$ and $Z(Q)$ contains more than $D^2$ lines, then Bezout's theorem, Theorem \ref{bezoutlines}, implies that $Q$ and $\partial_i Q$ have a common factor.  Since $Q$ is irreducible, $Q$ must divide $\partial_i Q$.  Since $\Deg \partial_i Q < \Deg Q$, it follows that $\partial_i Q$ is identically zero for each $i$.  This implies that $Q$ is constant.  By assumption, $Q$ is not the zero polynomial and so $Z(Q)$ is empty.  But we assumed that $Z(Q)$ contains at least $D^2 + 1$ lines, giving a contradiction.
\end{proof}

This finishes the proof of Lemma \ref{onebadpointZ(Q)} assuming Lemma \ref{lemV_p}.  
It only remains to prove Lemma \ref{lemV_p}.  

First we check that each point $x \in \RR^3$ lies in exactly one of the lines $\{ l_{p,q} \}_{q \in \RR^2}$.  Suppose $p = (p_1, p_2)$ and $q=(q_1, q_2)$ are points in $\RR^2$.  Using Equation \ref{eqlinex} and \ref{eqliney}, we see that $(x_1, x_2, x_3) \in l_{p,q}$ if and only if:

\begin{equation} \label{eqlinex'} 
2 x_1 = (p_1 + q_1) + (p_2 - q_2) x_3.
 \end{equation}

\begin{equation} \label{eqliney'} 
2 x_2 = (p_2 + q_2) + (q_1 - p_1) x_3.
 \end{equation}

We can rewrite these equations as a matrix equation for $q$ as follows:

$$\left( \begin{array}{cc}
1 & - x_3 \\
x_3 & 1 \end{array} \right)   \left( \begin{array}{cc} q_1 \\ q_2 \end{array} \right) = \left( 2x_1 - p_1 - x_3 p_2, 2 x_2 - p_2 + p_1 x_3    \right) =: a_p(x), $$

Note that $a_p(x)$ is a vector whose entries are polynomials in $x, p$ of degree $\le 1$ in $x$ and degree $\le 1$ in $p$.  Since the determinant of the matrix on the left-hand side is $1 + x_3^2 > 0$, we can uniquely solve this equation for $q_1$ and $q_2$.  The solution has the form 

\begin{equation} \label{qfromxp} q_1 = (x_3^2 + 1)^{-1} b_{1,p}(x); q_2 = (x_3^2 + 1)^{-1} b_{2,p}(x), \end{equation}

\noindent where $b_1, b_2$ are polynomials in $x, p$ of degree $\le 2$ in $x$ and degree $\le 1$ in $p$.  

We have now proven that each point of $\RR^3$ lies in a unique line from the set $\{ l_{p,q} \}_{q \in \RR^2}$.  Now we can construct the vector field $V_p$.  From Equations \ref{eqlinex'} and \ref{eqliney'}, we see that the vector $(  {p_2 - q_2}, {q_1 - p_1}, 2)$ is tangent to $l_{p,q}$.  If $x \in l_{p,q}$, then we can use Equation \ref{qfromxp} to expand $q$ in terms of $x, p$, and we see that the following vector field is tangent to $l_{p,q}$ at $x$:

$$ v_p(x) := (p_2 - (x_3^2+1)^{-1} b_{2,p}(x), (x_3^2+1)^{-1} b_{1,p}(x) - p_1, 2). $$

The coefficients of $v_p(x)$ are not polynomials because of the $(x_3^2+1)^{-1}$.  We define $V_p(x) = (x_3^2 +1) v_p(x)$, so

$$ V_p(x) = \left( p_2 (x_3^2+1) - b_{2,p}(x), b_{1,p}(x) - p_1( x_3^2 + 1), 2 x_3^2 + 2 \right). $$

The vector field $V_p(x)$ is tangent to the family of lines $\{ l_{p,q} \}_{q \in \RR^2}$.  Moreover, $V_p$ never vanishes because its last component is $2 x_3^2 +2$.  Therefore, the integral curves of $V_p$ are exactly the lines $\{ l_{p,q} \}_{q \in \RR^2}$.  Moreover, each component of $V_p$ is a polynomial of degree $\le 2$ in $x$ and degree $\le 1$ in $p$.

This finishes the proof of Lemma \ref{lemV_p} and hence the proof of Lemma \ref{noncluster}.


\begin{thebibliography}{5}

\vskip.125in



\bibitem[CEGSW]{CEGSW} K.L. Clarkson, H. Edelsbrunner, L. Guibas, M Sharir, and E. Welzl,
{ Combinatorial Complexity bounds for arrangements of curves and spheres}, Discrete Comput. Geom.
(1990)  5,  99-160.


\bibitem[E]{E} P. Erd{\H o}s, { On sets of distances of n points}, Amer. Math. Monthly (1946) 53, 248-250.

\bibitem[EKS]{EKS} Gy. Elekes, H. Kaplan, and M. Sharir, { On lines, joints, and incidences in three dimensions},
Journal of Combinatorial Theory, Series A (2011) 118, 962-977.

\bibitem[ES]{ES} Gy. Elekes and M. Sharir, { Incidences in three dimensions and distinct distances
in the plane}, Combinat. Probab. Comput., 20 (2011), 571-608.




\bibitem[GK]{GK} L. Guth and N. H. Katz, { Algebraic methods in discrete analogues of the Kakeya
problem}, Adv. in Math. (2010)  225,  2828-2839.

\bibitem[GK2]{GK2} L. Guth and N. H. Katz, { On the Erd{\H o}s distinct distances problem in the plane},  arXiv:1011.4105, accepted for publication in Annals of Math.

\bibitem[KMS]{KMS} H. Kaplan, J. Matou\u{s}ek, and M. Sharir, { Simple proofs of classical theorems in discrete geometry 
via the Guth--Katz polynomial partitioning technique}, Discrete Comput. Geom. 48 (2012), 499-517.




\bibitem[Mi]{Mi} J. Milnor, On the Betti numbers of real varieties, Proc. AMS 15, (1964) 275-280.


\bibitem[OP]{OP} O. A. Oleinik, I. B. Petrovskii, On the topology of real algebraic surfaces, Izv. Akad. Nauk SSSR
13, (1949) 389-402.




\bibitem[SS]{SS} M. Sharir and N. Solomon, Incidences between points and lines in four dimensions, Proc. 30th ACM Symp. on Computational Geometry (2014), to appear. 


\bibitem[SoTa]{SolTao} J. Solymosi and T. Tao, { An incidence theorem in higher dimensions}, Discrete Comput. Geom. Discrete Comput. Geom. 48 (2012), no. 2, 255-280. 


\bibitem[StTu]{ST}  Stone, A.H. and Tukey, J. W. , { Generalized sandwich theorems},
Duke Math. Jour. (1942)  9,  356-359.

\bibitem[Sz]{Sz} L. Sz\'ekely, { Crossing numbers and hard Erd{\H o}s problems in discrete geometry},  Combin. Probab. Comput.  (1997) 6  no. 3, 353-358.

\bibitem[SzTr]{SzTr} E. Szemer\'edi and W. T. Trotter Jr., { Extremal Problems in Discrete Geometry},
Combinatorica (1983)  3, 381-392.


\bibitem[Th]{Th} R. Thom, Sur l'homologie des vari\'et\'es alg\'ebriques r\'eelles, Differential and Combinatorial Topology, (Symposium in Honor of Marston Morse), Ed. S.S. Cairns, Princeton Univ. Press, (1965) 255-265.

\bibitem[VW]{VW} B. L. van der Waerden, {\it Modern Algebra}, Frederick Ungar Publishing Co., New York, 1950.

\end{thebibliography}
\end{document}